\documentclass[noams]{compositio}
\usepackage{amssymb}
\usepackage{amsmath}
\usepackage{lineno,hyperref} 
\usepackage{epsfig}  		% For postscript
\usepackage{epic,eepic}
 \usepackage{enumerate}
 \usepackage{color}
\usepackage{cleveref}
\usepackage{tikz}
\usetikzlibrary{arrows.meta,calc}

\definecolor{sectorfill}{RGB}{244,205,205}
\definecolor{sectorline}{RGB}{168,60,60}
\definecolor{actualfill}{RGB}{218,145,145}
\definecolor{halfplanefill}{RGB}{245,223,210}
\definecolor{targetgray}{RGB}{232,235,238}

\definecolor{modelblue}{RGB}{205,226,244}
\definecolor{deepblue}{RGB}{111,168,211}
\definecolor{actualblue}{RGB}{43,103,152}
\newtheorem{theorem}{Theorem}[section]
\newtheorem{proposition}[theorem]{Proposition}
\newtheorem{lemma}[theorem]{Lemma}

\newtheorem{question}[theorem]{Question}

\theoremstyle{remark}

\Crefname{theorem}{Theorem}{Theorems}

\newcommand{\C}{\mathbb C}
\newcommand{\D}{\mathbb D}
\newcommand{\R}{\mathbb R}
\newcommand{\T}{\mathbb T}
\newcommand{\Hh}{\mathbb H}
\newcommand{\Bcal}{\mathcal B}
\newcommand{\Bmer}{\mathcal B_{\mathrm{mer}}}

\newcommand{\Lcal}{\mathcal L}
\newcommand{\sing}{\operatorname{sing}}
\newcommand{\dist}{\operatorname{dist}}
\newcommand{\J}{\mathcal{J}}
\newcommand{\F}{\mathcal{F}}

\newcommand{\dd}{\,\mathrm d}

\renewcommand{\contentsname}{Contents\\{\footnotesize\normalfont(A table
of contents should normally not be included)}}
\begin{document}

\title[Julia limiting directions for \(\Bcal\)]{The inverse problem for Julia limiting directions
in the Eremenko--Lyubich class}
\author{Jiaxing Huang}
\email{hjxmath@szu.edu.cn}
\address{School of Mathematical Sciences, Shenzhen University, 518061,  China}
%\address{School of Mathematical Sciences, Shenzhen University, Shenzhen 518060, P.R. China}
\author{Jun Wang}
\email{majwang@fudan.edu.cn}
\address{School of Mathematical Sciences, Fudan University, Shanghai, 200433, China.}

\author{Xiao Yao}
\email{yaoxiao@nankai.edu.cn}
\address{School of Mathematical Sciences, Nankai University,  China}
%\dedication{A dedication can be included here.}
%At least one subject code is required. Please refer to
%\url{http://www.ams.org/msc/} for a list of codes.}
\thanks{Huang was partially supported by the NSFC (No.12201420 and 12231013)
and the NSF of Guangdong Province (No. 2026A1515010851). Wang was partially supported by the NSFC (No.12471072) and National Key R$\&$D Program of
China Grant (No.2021YFA1003200).
Yao was supported by NSFC of China (No. 12371074) and the Fundamental Research Funds for the Central Universities  (No. 63263084).
}

\begin{abstract}
We study the inverse problem for Julia limiting directions in the Eremenko-Lyubich class $\mathcal B$, which is the class of transcendental entire functions with a bounded set of singular values, and its meromorphic generalization $\mathcal B_{mer}$. Every nonempty compact subset of the circle can be the set of Julia limiting directions of an entire $f\in\mathcal B$. For finite unions of nondegenerate closed arcs, the corresponding $f\in\mathcal B$, even with every finite order at or above an explicit bound determined by the arc lengths, can be attained. We also refine the series construction of the latter two authors to obtain $f\in \mathcal B_{mer}$ with any prescribed order and the set of Julia limiting directions.
\end{abstract}
\maketitle

\section{Introduction and main results}

Let $f$ be a transcendental meromorphic function in the complex plane $\C$. By $\F(f)$ we mean the Fatou set of $f$ which is defined to be the set of points $z\in\C$ such that
all iterates $\{f^n\}_n$ are defined and form a normal family on a neighbourhood of $z$.  The complement of $\F(f)$ in $\C$ is the Julia set $\J(f)$ of $f$.
Standard references and recent research progress in transcendental iteration
theory can be found in the surveys \cite{Bergweiler, Schleicher, bergweiler25, EremenkoLyubich1990}. 

We also need some notations which come from value distribution theory, see \cite{Goldberg,Hayman}. For example, $T(r,f)$ and $M(r,f)$ denote the Nevanlinna characteristic function and the maximum modulus of $f$, respectively. The order $\rho(f)$ and lower order $\lambda(f)$ of $f$ are defined by
\[
\rho(f):=\limsup_{r\to\infty}
\frac{\log T(r,f)}{\log r},
\qquad
\lambda(f):=\liminf_{r\to\infty}
\frac{\log T(r,f)}{\log r}
\]
where $T(r,f)$ can be replaced by $\log M(r,f)$ for entire functions $f$.   

Throughout this paper, we use $\widehat{\C}=\C\cup\{\infty\},\Hh=\{\operatorname{Re}\zeta>0\},\T=\R/(2\pi\mathbb Z)$ to denote the Riemann sphere, the right half plane and the unit circle, respectively.
In addition, we write $$D(z,r)=\{w\in\C:|w-z|<r\},\quad \D=D(0,1).$$

A direction
$\theta\in\T$ is called a \emph{Julia limiting
direction} of $f$ if there is a sequence $\{z_k\}$ in the Julia set $\J(f)$ of $f$ such that as $k\to\infty$, 
\[
             |z_k|\rightarrow\infty,
       \qquad \arg z_k\rightarrow\theta\pmod{2\pi}.
\]
This notion was first introduced by Qiao \cite{Qiao} in his study of the
radial distribution of Julia sets. It is different from Julia direction in the classical value distribution theory, which is defined in terms of the distribution of preimages of values in arbitrarily small angular domains.

By $\Lcal(f)$ we mean the set of all Julia limiting directions of $f$. This set is also referred to as the \emph{radial distribution} of $\J(f)$ (see, for
example, \cite{Fletcher2021,FletcherSteranka2025}). More generally, for an unbounded set $X\subset\C$, define the \emph{radial distribution} of $X$ by 
\begin{equation*}
 L(X):=
 \left\{\theta\in\T:
 \begin{array}{l}
 \exists z_k\in X\text{ with }|z_k|\to\infty,\
 \arg z_k\to\theta\pmod{2\pi}
 \end{array}\right\};
\end{equation*}
thus $\Lcal(f)=L(\J(f))$. Since \(\J(f)\) is unbounded, clearly \(\Lcal(f)\) is a nonempty closed subset of \(\mathbb T\). The inverse problem is to determine which nonempty closed subsets of \(\mathbb T\) can occur as \(\Lcal(f)\). 

The exponential family provides a simple illustration of the dependence of Julia limiting directions on the singular value. Let \(E_\lambda(z)=\lambda e^z\), where \(\lambda\neq0\). 
The second and third authors proved that if its singular value $0\in \mathcal F(E_\lambda)$, then $\Lcal(E_\lambda)=TD(E_\lambda)$,
whereas if $0\in \mathcal J(E_\lambda)$, $
\Lcal(E_\lambda)=\mathbb T$, see \cite[Example~1]{WangYao}. 
Thus, even in the Speiser class, the set of Julia limiting directions may depend essentially on the position of a singular value.

For a transcendental entire
function $f$ of finite lower order $\mu$, Qiao \cite{Qiao}
proved a strong growth obstruction that $\Lcal(f)$ contains a closed arc $I$
such that its Lebesgue measure
\begin{equation}\label{eq:qiao-intro}
                   |I|\geq\frac{\pi}{\max\{1/2,\mu\}}.
\end{equation}
Thus, $\Lcal(f)=\mathbb T$ if $\mu\leq 1/2$, while a finite set, a Cantor set, or any compact
set containing no nondegenerate arc can be only realized by an entire function of infinite lower order. 

The measure of $L(f)$ for meromorphic $f$ has also been studied; see, for example, Zheng, Wang and Huang \cite{ZhengWangHuang}. Wang and Chen \cite{WangChen} investigated Julia
limiting directions of entire solutions to complex
differential equations.

As to the inverse problem, Wang and Yao \cite[Theorem~2]{WangYao} proved that every nonempty compact set
$E\subset\T$ is equal to $\Lcal(f)$ for some entire function
$f$ of infinite lower order. They also constructed entire functions with finite order for certain finite unions of arcs whose widths satisfy the restriction suggested by \eqref{eq:qiao-intro}. Their meromorphic case is even more flexible: for every nonempty compact
$E\subset\T$ and every $\rho\in[0,\infty]$, there is a transcendental
meromorphic function $g$ of order $\rho$ such that
$
                         \Lcal(g)=E.
$
Later, related problems concerning Julia limiting directions
were also studied in higher dimensional quasiregular dynamics;
see \cite{Fletcher2021,FletcherSteranka2025}.

For a transcendental entire function \(f\), write
\[
S(f):=\overline{\operatorname{sing}(f^{-1})}
\]for the closure of its critical and finite asymptotic values. The Eremenko--Lyubich and Speiser classes are defined by
\[
\mathcal B:=\{f:S(f)\text{ is bounded}\}
\quad\text{and}\quad
\mathcal S:=\{f:S(f)\text{ is finite}\},
\]
respectively. 
These classes play a central role in transcendental dynamics because singular values are the obstructions to the continuation of inverse branches and largely govern the global dynamics. For \(f\in\mathcal B\), the behavior near infinity can be described by logarithmic tracts over the complement of a sufficiently large disk. The Speiser class \(\mathcal S\subset\mathcal B\) imposes the stronger condition that only finitely many singular values occur.

It remains unclear whether the entire functions obtained in the construction of \cite{WangYao} belong to $\mathcal B$ or $\mathcal S$. As we know, there are no references that have settled the inverse problem of Julia limiting directions under a restriction on singular values. This motivates the following question.
\begin{question}
Given a nonempty compact set \(E\subset\mathbb T\), does there exist a transcendental entire function \(f\in\mathcal B\) or \(f\in\mathcal S\)  such that
\(
\Lcal(f)=E?
\)
\end{question}

We answer Question~1.1 affirmatively for \(\mathcal B\). The corresponding problem for \(\mathcal S\) remains open, even when \(E\) is finite. Two properties of the class $\Bcal$, which was proved by Eremenko and Lyubich \cite{EL}, are particularly relevant. One property is that for \(f\in\mathcal B\), its escaping set $I(f)\subset \mathcal J(f)$, where 
\(I(f):=\{z\in\mathbb C:f^n(z)\to\infty\}\). The other is that every component of \(f^{-1}(\{w:|w|>R\})\) is a logarithmic tract for sufficiently large \(R\). This tract structure forms the basis of our construction. 
\begin{theorem}\label[theorem]{thm:compact-main}
Let \(E\subset\T\) be nonempty and compact. Then there exists a transcendental entire function \(f\in\Bcal\) such that $\Lcal(f)=E.$

\end{theorem}

Theorem \ref{thm:compact-main} gives a geometric realization of arbitrary compact direction sets, but the construction based on Rempe's approximation theorem generally yields functions of infinite lower order. For the case of finite order, Qiao's theorem shows that compact direction sets must contain at least one nondegenerate arc. Thus, it is natural to consider the finite order case for direction sets consisting of finitely many nondegenerate closed arcs. For this purpose, we employ Bishop's model theorem, which allows us to control the growth of the resulting functions.

\begin{theorem}\label[theorem]{thm:arcs-main}
Let $E\subset\T$ be a nonempty compact set, the connected components of which are nondegenerate closed arcs $I_j\ (j=1,2,\ldots,m)$ of length $b_j\in(0, 2\pi]$. 
Let
\[
 \rho_E:=\max_{1\leq j\leq m}\frac{\pi}{b_j}.
\]
Then, for every $\rho\in[\rho_E,\infty)$, there exists a
transcendental entire function $f\in\Bcal$ such that
$\Lcal(f)=E$ and $\lambda(f)=\rho(f)=\rho$.
\end{theorem}

For $\rho>1/2$, \Cref{thm:arcs-main} strengthens
\cite[Theorem~3]{WangYao} by allowing the realizing
function to belong to $\Bcal$. The endpoint
$\rho=1/2$, which occurs only when $E=\T$,
is also included.
 
For meromorphic functions, large values may occur in
bounded components surrounding poles. By controlling 
the positions and multiplicities of these poles, the second and third authors \cite{WangYao} realize any nonempty compact direction set at any given order.

Following Bergweiler and Kotus
\cite{BergweilerKotus}, let $\Bmer$ denote the class of transcendental
meromorphic functions $g$ with bounded $S_{\mathrm{fin}}(g)$, where  
\[
 S_{\mathrm{fin}}(g):=
 \overline{\sing(g^{-1})\setminus\{\infty\}}.
\]
The closure is taken in $\C$, and $\infty$ is allowed to be a singular value.
Our meromorphic realizations corresponding to \cite[Theorem 1]{WangYao} can be chosen in \(\Bmer\) by reconsidering the meromorphic series constructed. 

\begin{theorem}\label[theorem]{thm:meromorphic-main}
Let $E\subset\T$ be nonempty and compact, and let
$\rho\in[0,\infty]$. Then there exists a transcendental meromorphic function
$g\in\Bmer$ such that
\(
                       \rho(g)=\rho
 \) and
 \(\Lcal(g)=E.
\)
\end{theorem}

The paper is organized as follows. 
\begin{itemize}
\item Section~\ref{sec:pre} records some elementary
properties of radial distributions, provides a
criterion for Julia limiting directions, and
constructs locally finite families of half strips
used to realize compact direction sets.

\item Section~\ref{sec:Remp} develops an approximation argument
based on Rempe's theorem and establishes the exact decomposition of the preimage $\{w:|w|>R\}$ into logarithmic tracts. 

\item In Section~\ref{sec:inv}, we apply
the construction in Section~\ref{sec:Remp} to prove \Cref{thm:compact-main}. 

\item Sections~\ref{sec:B} and~\ref{sec:realB}
are devoted to the finite order case. Section~\ref{sec:B} records the
form of Bishop's realization theorem that we use, treats power sector
models, and establishes estimates for strips of variable width.
Section~\ref{sec:realB} uses these models to prove
\Cref{thm:arcs-main}.

\item Section~\ref{sec:mero} contains the proof of \Cref{thm:meromorphic-main}
\end{itemize}

\paragraph{\bf Acknowledgment.}
This work was supported by the Visiting Scholar Program
of the Chern Institute of Mathematics at Nankai University
and was completed during the first author's visit to the
School of Mathematical Sciences at Fudan University.
The first author thanks both institutions for their hospitality.
The authors also thank Christopher Bishop and Walter Bergweiler
for helpful discussions and suggestion for the constructions. 
%suggestion and  discussion for the preparation of this paper.  

\paragraph{\bf AI Disclaimer:} The main ideas and original proof strategies for 
\Cref{thm:compact-main}, \Cref{thm:arcs-main} and \Cref{thm:meromorphic-main} were developed by the authors without AI assistance and were already present in an earlier version of the manuscript.
 ChatGPT was subsequently used to explore and refine certain technical arguments, as well as to improve the language and exposition.

\section{Some lemmas on radial distributions}\label{sec:pre}
We begin with some preliminary results on radial distributions and the construction of disjoint half strips. Next two lemmas are directly from the definitions of radial distribution and half strip, respectively.
\begin{lemma}\label[lemma]{lem:LD-compact}
If $X\subset\C$ is unbounded, then $L(X)$ is a nonempty compact subset
of $\T$.
\end{lemma}
\begin{lemma}\label{lem:one-strip}
For $R,a>0$ and $\theta\in\T$, a half strip 
\[
 \Sigma(\theta;R,a):=
 \{e^{i\theta}(x+iy):x>R,\ |y|<a\}
\] must have $L(\Sigma(\theta;R,a))=\{\theta\}.$
\end{lemma}

For an entire function $f$, its fast escaping set is defined by
\[
 A(f):=\{z:\text{there exists }L\in\mathbb N\text{ such that }
          |f^{n+L}(z)|\geq M^n(R,f)\text{ for every }n\in\mathbb N\}.
\]
Here $M^n(r,f)$ denotes iteration of $r\mapsto M(r,f)$, and $R>0$
may be any value for which $M^n(R,f)\to\infty$ as $n\to\infty$.

With the notation of the radial distribution $L(X)$ of $X$, we set 
\[
 \Lcal_I(f):=L(I(f)),\qquad
 \Lcal_A(f):=L(A(f)).
\]
If \(f\in\Bcal\), then
\(
A(f)\subset I(f)\subset \J(f),
\)
and \(A(f)\) is dense in \(\J(f)\) (see \cite[Theorem 1]{EL} and \cite[Theorem~5.1]{RipponStallard}). We therefore obtain the following result.

\begin{proposition}\label[proposition]{prop:direction-equality-B}
\(
                 \Lcal_A(f)=\Lcal_I(f)=\Lcal(f)
\) for $f\in\Bcal$.
\end{proposition}

The next proposition which was suggested by Walter Bergweiler,  allows us to obtain Julia limiting directions from
sequences on which the function tends to infinity.

\begin{proposition}\label[proposition]{prop:large-values-julia}
Let $f\in\Bcal$, and let $\{z_n\}$ be a sequence tending to infinity and
$|f(z_n)|\to\infty$ as $n\to\infty$. Then
\begin{equation}\label{eq:large-values-proximity}
 \frac{\dist(z_n,\J(f))}{|z_n|}\longrightarrow0,\quad \text{as}\,\,n\to\infty.
\end{equation}
In particular, if $\arg z_n\to\theta$, then $\theta\in\Lcal(f)$.
\end{proposition}

\begin{proof}
Choose
$$R>\max\left\{|f(0)|,\sup_{s\in S(f)}|s|\right\}.$$ 
The logarithmic derivative estimate of Eremenko and Lyubich \cite[Lemma~1]{EL} gives
\begin{equation}\label{eq:EL-log-derivative}
 \left|\frac{zf'(z)}{f(z)}\right|
 \geq\frac1{4\pi}\log\frac{|f(z)|}{R}
 \qquad\text{when }|f(z)|>R.
\end{equation}
Let
$z\in\F(f)$ and $d=\dist(z,\J(f))$. Consider the Fatou
component $V$ containing $f(D(z,d))$. As $f\in\Bcal$, it follows from \cite[Proposition~3]{EL} that $V$ is simply connected. By the Schwarz--Pick and Koebe distortion
theorems, we have
\[
 \lambda_V(f(z))|f'(z)|\leq\frac2d\qquad\text{and}\qquad
 \lambda_V(f(z))\geq
 \frac1{2\dist(f(z),\partial V)}.
\]
Let $a_0\in\J(f)$. Then
$$\dist(f(z),\partial V)\leq|f(z)-a_0|,$$ and hence,
\[
 \left|\frac{zf'(z)}{f(z)}\right|
 \leq\frac{4|z|}{d}
       \left(1+\frac{|a_0|}{|f(z)|}\right).
\]
Combining this inequality with \eqref{eq:EL-log-derivative} yields, for $|f(z)|>R$,
\begin{equation*}
 \frac{d}{|z|}=\frac{\dist(z,\J(f))}{|z|}
 \leq4\left|\frac{f(z)}{zf'(z)}\right|\left(1+\frac{|a_0|}{|f(z)|}\right)\le
 \frac{16\pi}{\log(|f(z)|/R)}\left(1+\frac{|a_0|}{|f(z)|}\right).
\end{equation*}
The same estimate is immediate when $z\in\J(f)$, so
\eqref{eq:large-values-proximity} follows.

Choose $\xi_n\in\J(f)$ such that
\[
|\xi_n-z_n|=\dist(z_n,\J(f)).
\]
By \eqref{eq:large-values-proximity}, we have
\[
\left|\frac{\xi_n}{z_n}-1\right|
=\frac{\dist(z_n,\J(f))}{|z_n|}
\longrightarrow 0.
\]
Hence, $|\xi_n|\to\infty$ and
\[
\arg\xi_n-\arg z_n\to 0\ (\text{mod}\ 2\pi).
\]
Thus, if $\arg z_n\to\theta$, then
$\arg\xi_n\to\theta$, proving that $\theta\in \Lcal(f)$.
\end{proof}

For $\theta_1,\theta_2\in\T$, define
\[
d_{\T}(\theta_1,\theta_2)
:=\min_{k\in\mathbb Z}|\theta_1-\theta_2+2\pi k|.
\] This expression is independent of the representatives of the two angles.

A family \(\mathcal U\) of subsets of \(\mathbb C\) is called locally
finite  if
\[
\#\{\Sigma\in\mathcal U:K\cap\Sigma\neq\varnothing\}<\infty
\]
for every compact set \(K\subset\mathbb C\).
The following observation will be used for the half strip models.
\begin{lemma}\label[lemma]{lem:compact-strips}
Let $E\subset\T$ be nonempty, compact, and infinite. Then, there are distinct
points $\theta_n\in E$ and pairwise disjoint half strips
\(
 \Sigma_n=e^{i\theta_n}\{x+iy:x>R_n,\ |y|<1\}\)
 with \(R_n\to\infty,
\)
such that the family $\{\Sigma_n\}$ is locally finite, and
\begin{equation*}
        L\left(\bigcup_{n\geq1}\Sigma_n\right)=E.
\end{equation*}
\end{lemma}

\begin{proof}
Choose a countable dense subset $\{\theta_n:n\geq1\}$ of $E$ with distinct
elements. Inductively choose $R_n\geq n+2$ so large that $\Sigma_n$ is
disjoint from $\Sigma_1,\ldots,\Sigma_{n-1}$. This is possible because two
rays with distinct arguments separate linearly at infinity. Since every
point of $\Sigma_n$ has modulus at least $R_n$, the condition $R_n\to\infty$
implies local finiteness.

We take a sequence $\{z_k\}$ with $z_k\in\bigcup_n\Sigma_n$ and $|z_k|\to\infty$. If infinitely many
$z_k$ lie in one fixed $\Sigma_n$, every corresponding angular limit is
$\theta_n\in E$. Otherwise, after passing to a subsequence if necessary, we may assume that $z_k\in\Sigma_{n_k}$ with
$n_k\to\infty$. Then,
\[
 d_{\T}(\arg z_k,\theta_{n_k})
 \leq\arctan\frac{1}{R_{n_k}}\longrightarrow0.
\]
Since $\theta_{n_k}\in E$ and $E$ is closed, every limiting direction of $(z_k)$ belongs to $E$.
Hence, $$L\left(\bigcup_{n\geq1}\Sigma_n\right)\subset E.$$

Now, let $\theta\in E$. Since $\{\theta_n:n\geq 1\}$
is dense in $E$, we may choose a subsequence $\{\theta_{n_k}\}$ such that
$\theta_{n_k}\to\theta$ in $\T$. Clearly, the points $$z_k:=(R_{n_k}+k)e^{i\theta_{n_k}}\in\Sigma_{n_k},$$ satisfy $|z_k|\to\infty$ and
$\arg z_k\to\theta$ (mod $2\pi$). Hence
\[
\theta\in L\left(\bigcup_{n\geq 1}\Sigma_n\right).
\]
This completes the proof of \Cref{lem:compact-strips}.
\end{proof}

\section{Rempe's approximation models}\label{sec:Remp}
In this section, we introduce Rempe's approximation theorem and apply it to construct entire functions in \(\mathcal B\) with logarithmic tracts contained in prescribed half strips. By taking a suitably weighted sum of the constructed entire functions, we obtain a concrete control over the preimages of sufficiently large values. This construction will be used in Section~\ref{sec:inv} to prove \Cref{thm:compact-main}.

Recall that a \emph{logarithmic tract of $f$ over $\mathbb{C}\setminus \overline{D(0, R)}$} is a
component $V$ of $f^{-1}(\mathbb{C}\setminus \overline{D(0, R)})$ for which
\[
              f:V\longrightarrow \mathbb{C}\setminus \overline{D(0, R)}
\]
is a universal covering.

Let
\(
 \widehat H:=\{x+iy:x>-14\log^+|y|\}.
\)
The main ingredients in the proof of
\Cref{thm:compact-main} are the following
two results of Rempe \cite{Rempe2014}.

\begin{theorem}[{\cite[Theorem 1.7]{Rempe2014}}]\label[theorem]{thm:rempe}
Let $T\subset\C$ be an unbounded simply connected domain and let
$\Psi:T\to\widehat H$ be a conformal isomorphism satisfying
\begin{equation}\label{eq:model-end}
 \Psi(z_k)\to\infty\text{ in }\widehat H
 \quad\Longrightarrow\quad
 z_k\to\infty\text{ in } T.
\end{equation}
Set $g:=\exp\circ\Psi$. Then there exists an entire function
$G\in\Bcal$ such that, as $z\to\infty$,
\begin{equation*}
G(z)=
\begin{cases}
 g(z)+O(1/z),\quad &z\in T,\\[4pt]
 O(1/z),\quad &z\notin T.
 \end{cases}
\end{equation*}
\end{theorem}

\begin{proposition}[{\cite[Proposition~2.2]{Rempe2014}}]\label[proposition]{prop:stability}
Let $U$ be an unbounded simply connected domain and $\psi:U\to\Hh$ be conformal, put $h=e^\psi$, and assume that
$\psi(z_k)\to\infty$ implies $z_k\to\infty$. Suppose that
$q:U\to\C\setminus\{0\}$ is holomorphic and
\[
                         |q-h|\leq1\quad\text{on }U.
\]
Then
\[
 U^0:=\{z\in U:|q(z)|>2\}
\]
is simply connected and
\[
 q:U^0\longrightarrow\{w:|w|>2\}
\]
is a universal covering.
\end{proposition}

For each half strip used below, its closure in the Riemann sphere, as well as the closure of $\widehat H$
is a Jordan domains. We choose the conformal
isomorphism so that \eqref{eq:model-end} holds, i.e, the prime end at infinity is sent to the prime end at
infinity by Caratheodory's theorem.
Thus, we give the following approximation models for proving \Cref{thm:compact-main}.
\begin{theorem}
\label[theorem]{thm:weighted}
Let $\mathcal I$ be a nonempty finite or countably infinite index set. Suppose that
$\{\Sigma_n:n\in\mathcal I\}$ is a pairwise disjoint, locally finite family
of half strips and
\[
               \Psi_n:\Sigma_n\longrightarrow\widehat H
\]
is a conformal isomorphism satisfying \eqref{eq:model-end}. Then there
exist a transcendental entire function $F\in\Bcal$, a number $A>0$, and
domains $V_n\subset\Sigma_n$ such that
\begin{equation}\label{eq:exact-large-preimage}
 F^{-1}(\{w:|w|>2A\})=\bigcup_{n\in\mathcal I}V_n,
\end{equation}
and
\begin{equation*}
 F:V_n\longrightarrow\{w:|w|>2A\}
\end{equation*}
is a universal covering for every $n\in\mathcal I$.
\end{theorem}

\begin{proof}
Let $g_n:=e^{\Psi_n}$. By \Cref{thm:rempe}, for each $n$, 
there exist an entire function $G_n\in\Bcal$ and a
constant $C_n\geq 1$ such that
\begin{equation}\label{eq:global-out}
\begin{cases}
|G_n(z)-g_n(z)|\leq \dfrac{C_n}{|z|^+},
    & z\in\Sigma_n,\\[6pt]
|G_n(z)|\leq \dfrac{C_n}{|z|^+},
    & z\notin\Sigma_n,
\end{cases}
\end{equation}
where $|z|^+:=\max\{1,|z|\}$.

Choose $0<\varepsilon_n\leq1$ (one may take
$\varepsilon_n=2^{-n}/C_n$) so that
\begin{equation}\label{eq:weight-sum}
                   C:=\sum_{n\in\mathcal I}\varepsilon_nC_n<\infty.
\end{equation}
Define
\[
                         F:=\sum_{n\in\mathcal I}\varepsilon_nG_n.
\]
 For any compact $K\subset\C$, the local finiteness of the family $\{\Sigma_n\}_n$ implies that 
\[
 \mathcal I_K:=\{n:K\cap\Sigma_n\neq\varnothing\}
\]
is finite. Thus, $$\sum_{n\in\mathcal I_K}\varepsilon_nG_n<\infty,$$
while for
$n\notin\mathcal I_K$, \eqref{eq:global-out} gives
\[
        \sup_{z\in K}|\varepsilon_nG_n(z)|\leq\varepsilon_nC_n.
\]
The remaining sum converges uniformly on $K$ by
\eqref{eq:weight-sum}. Hence $F$ is entire.

Now, fix $n$ and let $z\in\Sigma_n$. Then,
$z\notin\Sigma_k$ for $k\neq n$, and therefore
\begin{equation}\label{eq:dominant}
 \left|F(z)-\varepsilon_ng_n(z)\right|
 \leq\varepsilon_n|G_n(z)-g_n(z)|
   +\sum_{k\neq n}\varepsilon_k|G_k(z)|
 \leq\frac{C}{|z|^+}\leq C.
\end{equation}
If $z\notin\bigcup_n\Sigma_n$, then
\begin{equation}\label{eq:outside-all}
                         |F(z)|\leq\frac{C}{|z|^+}\leq C.
\end{equation}

Choose $A>C+1$ and consider
\[
 U_n:=\{z\in\Sigma_n:\varepsilon_n|g_n(z)|>A\}.
\]
The function
\(
 \psi_n:=\Psi_n+\log\varepsilon_n-\log A
\)
maps $U_n$ conformally onto $\Hh$ and satisfies the condition \eqref{eq:model-end}.

Let
\[
                    h_n:=\frac{\varepsilon_ng_n}{A}=e^{\psi_n},
 \qquad q_n:=\frac{F}{A}.
\]
By \eqref{eq:dominant}, we have 
\[
                  |q_n-h_n|\leq\frac{C}{A}<1, \qquad\text{on}\ U_n.
\]
Moreover, $q_n$ has no zero on $U_n$, since
\[
 |F(z)|\geq\varepsilon_n|g_n(z)|-C>A-C>0.
\]
Applying \Cref{prop:stability} gives a simply connected domain
\[
 V_n:=\{z\in U_n:|F(z)|>2A\}
\]
on which $F$ is a universal covering.

It remains to show that there are no additional components of \(F^{-1}(\{w:|w|>2A\})\) other than the \(V_n\). If
$z\in\Sigma_n\setminus U_n$, then \eqref{eq:dominant} gives
\[
                 |F(z)|\leq A+C<2A.
\]
Together with \eqref{eq:outside-all}, this proves
\eqref{eq:exact-large-preimage}. Thus
\[
 F:F^{-1}(\{w:|w|>2A\})\longrightarrow\{w:|w|>2A\}
\]
is a covering map. Consequently,
\[
                          S(F)\subset\overline{D(0,2A)}.
\]
Hence $F\in\Bcal$. Finally, a polynomial has no universal covering tract
over infinity, so $F$ is transcendental. This completes the proof of \Cref{thm:weighted}.
\end{proof}

\section{Proof of the compact inverse problem}\label{sec:inv}
In this section, we would like to show the realization of arbitrary compact set for disjoint type entire functions in the Eremenko-Lyubich class.
\begin{proof}[Proof of \Cref{thm:compact-main}]
Let $\{\Sigma_n\}_{n\in\mathcal I}$ be the family of half strips and \[
                         \Omega:=\bigcup_{n\in\mathcal I}\Sigma_n.
\]
Using \Cref{lem:one-strip} and \Cref{lem:compact-strips}, we always have $L(\Omega)=E$.

For each $n$, we can take a conformal isomorphism
$\Psi_n:\Sigma_n\to\widehat H$ satisfying \eqref{eq:model-end}. Hence, applying
\Cref{thm:weighted}, we obtain a transcendental entire function $F\in\Bcal$, a constant $A>0$, and logarithmic tracts
$V_n\subset\Sigma_n$ over $\C\setminus \overline{D(0, 2A)}$.

Let $\eta>0$ be sufficiently small that 
\[
 2\eta A<\frac12,\qquad
 \eta M(1,F)<\frac12,\qquad
 \eta C<\frac12,
\]
where $C$ is the constant in \eqref{eq:weight-sum}. 
Let $f:=\eta F$. Since
\(
 S(f)=\eta S(F)\subset D(0,1/2),
 \) and \( 
 f(\overline\D)\subset D(0,1/2)\subset\D,
\) the entire function $f\in\Bcal$.

For every $R\geq1$, let
\[
 V_n(R):=\{z\in V_n:|f(z)|>R\}.
\]
Since $2\eta A<1/2<R$ and
$F:V_n\to\{w:|w|>2A\}$ is a universal covering, it follows that $V_n(R)$ is connected and
\[
 f:V_n(R)\longrightarrow\mathbb{C}\setminus \overline{D(0, R)}
\]
is again a universal covering. The exact identity
\eqref{eq:exact-large-preimage} shows that the $V_n(R)$ are all the
components of $f^{-1}(\mathbb{C}\setminus \overline{D(0, R)})$.

If $z\notin\Omega$, then $$|f(z)|=\eta|F(z)|<\eta C<1/2.$$ Hence $z\in\mathcal F(f)$. Therefore
\[
                         \J(f)\subset\Omega.
\]
By \Cref{lem:one-strip} and \Cref{lem:compact-strips},
\begin{equation*}
                        \Lcal(f)=L(\J(f))\subset L(\Omega)= E.
\end{equation*}

For the reverse inclusion, choose $a\in A(f)$. After replacing $a$ by an
iterate, we may suppose that
\[
                         w:=f^N(a)\in A(f),\qquad |w|>1.
\]
For every $n$, the fiber
\[
                  \{z\in V_n(1):f(z)=w\}
\]
is an infinite and unbounded subset of $A(f)$, as $f:V_n(1)\to\{w:|w|>1\}$ is a universal covering and $A(f)$ is completely invariant
by \cite[Theorem~2.2]{RipponStallard}.

Fix $\theta\in E$. If $\theta=\theta_n$ is a central
direction, the unbounded fiber in
$V_n(1)\subset\Sigma_n$ contains a sequence in
$A(f)$ whose arguments tend to $\theta_n$.
Thus $\theta\in\Lcal_A(f)$.

If $\theta$ is not a central direction, then by
the density of $\{\theta_n\}$ in $E$,
we may choose a subsequence $\{\theta_{n_k}\}$ such that
$\theta_{n_k}\to\theta$. From the fiber in
$V_{n_k}(1)$, the points
\(
z_k=e^{i\theta_{n_k}}(x_k+iy_k),
\) with \(x_k>k\) and \(|y_k|<1\),
satisfy $z_k\in A(f)$ and
\[
d_{\T}(\arg z_k,\theta)
\leq\arctan\frac1k+
d_{\T}(\theta_{n_k},\theta)
\longrightarrow0.
\]
Thus $\theta\in\Lcal_A(f)$.

In either case, we have
$E\subset\Lcal_A(f)$. By
\Cref{prop:direction-equality-B},
\[
E\subset\Lcal_A(f)=\Lcal_I(f)
=\Lcal(f)\subset E.
\]
This proves $\Lcal(f)=E$.
\end{proof}

\section{Bishop's models for the Eremenko--Lyubich class}\label{sec:B}

In this section, we recall Bishop's models for the Eremenko--Lyubich class
and construct two types of models based on power sectors
and strips of variable width to treat the finite order case.
\subsection{Realization and normalization at infinity}

Let $$\Omega=\bigcup_j\Omega_j,$$ where the $\Omega_j$ are
pairwise disjoint, unbounded, simply connected domains in
$\C$ with connected boundaries, and distinct components
accumulate only at infinity.

For each $j$, let $\tau_j:\Omega_j\to\Hh$ be a conformal
isomorphism, and define $\tau:\Omega\to\Hh$ by
$\tau|_{\Omega_j}=\tau_j$.
Suppose that
\[
|\tau(z_n)|\to\infty
\quad\Longrightarrow\quad
|z_n|\to\infty
\]
for every sequence $\{z_n\}\subset\Omega$.
Following Bishop \cite{BishopB}, the pair
$(\Omega,\mathcal F)$, where
$\mathcal F=\exp\circ\tau$, is called a model. For such a model, we use the notation
\[
\Omega(s):=\{z\in\Omega:\operatorname{Re}\tau(z)>s\},
\quad
\Omega(s,t):=\{z\in\Omega:
s<\operatorname{Re}\tau(z)<t\},
\]
and write $\Omega_j(s):=\Omega_j\cap\Omega(s)$, for $0<s<t$.

We will use the quasiregular construction underlying
Bishop's realization theorem, together with the following
version of the Teichm\"uller--Wittich--Belinskii theorem;
see \cite[Theorem~2.5]{BishopS} and
\cite[Lemma~4.4]{BER2024}.

\begin{theorem}\label[theorem]{thm:twb}
Let $\psi:\C\to\C$ be quasiconformal. If, for some $R>0$,
\[
 \int_{|z|>R}\frac{|\mu_\psi(z)|}{|z|^2}\dd A(z)<\infty,
\]
then $\psi(z)=az+o(|z|)$ as $z\to\infty$, for some $a\neq0$.
\end{theorem}

The following proposition combines Bishop's construction with the above theorem to obtain a realization whose quasiconformal change of coordinates is asymptotic to the identity at infinity.
\begin{proposition}\label[proposition]{thm:bishop}
Suppose that $(\Omega,\mathcal F)$ is a model and
\begin{equation*}
 \int_{\Omega(1/2,2)\cap\{|z|>1\}}\frac{\dd A(z)}{|z|^2}<\infty.
\end{equation*}
Then, there are a transcendental entire function $f_0\in\Bcal$, a quasiregular
map $Q:\C\to\C$, and a quasiconformal homeomorphism $\phi:\C\to\C$
such that
\begin{align}
 Q&=f_0\circ\phi,&
 \phi(z)&=z+o(|z|)\quad(z\to\infty),\label{eq:model-normalization}\\
 Q&=\mathcal F\quad\text{on }\Omega(2),&
 |Q|&\leq e^2\quad\text{on }\C\setminus\Omega(2),\label{eq:model-agreement}\\
 S(f_0)&\subset\overline{D(0,e)}.&\label{eq:bishop-sing}
\end{align}
Moreover, $\mu_\phi=\mu_Q=0$ almost everywhere off $\Omega(1/2,2)$.
For every $R>e^2$,
\begin{equation}\label{eq:bishop-preimage}
 f_0^{-1}(\C\setminus\overline{D(0,R)})
       =\phi\bigl(\Omega(\log R)\bigr),
\end{equation}
and the restriction of $f_0$ to each component of this set is a
universal covering of $\C\setminus\overline{D(0,R)}$.
\end{proposition}

\begin{proof}
The result follows from Bishop's quasiregular construction
\cite[Theorem~1.1 and Sections~2--3]{BishopB}.
The finite logarithmic area condition allows us to apply
\Cref{thm:twb}, and a linear normalization then gives
$\phi(z)=z+o(|z|)$ as $z\to\infty$.
\end{proof}

Since $\phi(z)=z+o(|z|)$ as $z\to\infty$,
the map $\phi$ preserves limiting directions. In particular,
for every unbounded set $X\subset\C$,
\begin{equation}\label{eq:phi-directions}
L(\phi(X))=L(X).
\end{equation}

\subsection{Power sector models}\label{sec:power-models}
We first construct a model based on power sectors whose opening angles determine the order of the corresponding entire function. This model will be used to establish the case \(\rho=\rho_E\) of \Cref{thm:arcs-main}.

For each arc $I_j$ in \Cref{thm:arcs-main}, choose a central
direction $\theta_j$ and write
\[
I_j=\{\theta_j+t\pmod{2\pi}:|t|\leq b_j\},
\qquad
b_j=\frac{\ell_j}{2},\qquad
p_j=\frac{\pi}{2b_j}.
\]
Consider the sector
\[
S_j=\{e^{i\theta_j}re^{it}:r>0,\ |t|<b_j\},
\]
and define a conformal map $P_j:S_j\to\Hh$ by
\[
P_j(e^{i\theta_j}re^{it})=r^{p_j}e^{ip_jt}.
\]
When $b_j=\pi$, the domain $S_j$ is the plane slit
along the ray opposite to $e^{i\theta_j}$.

Fix $b>1$ and set
\[
T_j(s):=\{z\in S_j:\operatorname{Re}P_j(z)>b+s\},
\qquad s\geq0.
\]
Let $\Omega_j=T_j(0)$, $\tau_j=P_j-b$, and
$\Omega=\bigcup_{j=1}^m\Omega_j$.
Defining $\tau$ by $\tau|_{\Omega_j}=\tau_j$,
we obtain a model $(\Omega,\mathcal F)$,
where $\mathcal F=\exp\circ\tau$ and
\[
\Omega(s)=\bigcup_{j=1}^m T_j(s).
\]

We now verify that the domains $T_j(s)$ have
the prescribed limiting directions.
\begin{lemma}\label[lemma]{lem:sector-direction}
For every $s\geq0$ and every $j$, we have $L(T_j(s))=I_j$ and  $L(\Omega(s))=E$.
\end{lemma}

\begin{proof}
The inclusion $L(T_j(s))\subset I_j$ follows from $T_j(s)\subset S_j$.
Conversely, for $|\delta|<b_j$, let
$z_r=re^{i(\theta_j+\delta)}$. Then
\[
\operatorname{Re}P_j(z_r)
=r^{p_j}\cos(p_j\delta)\longrightarrow\infty
\qquad (r\to\infty).
\]
Hence $z_r\in T_j(s)$ for all sufficiently large $r$,
so $\theta_j+\delta\in L(T_j(s))$.
By the closedness of $L(T_j(s))$, we obtain
$L(T_j(s))=I_j$.
Finally, since $\Omega(s)=\bigcup_{j=1}^m T_j(s)$,
we have
\[
L(\Omega(s))
=\bigcup_{j=1}^m L(T_j(s))
=\bigcup_{j=1}^m I_j=E.
\]
\end{proof}

We next verify that the power-sector model satisfies the finite logarithmic area condition required in \Cref{thm:bishop}.
\begin{lemma}\label[lemma]{lem:logarea}
The power sector model $(\Omega,e^\tau)$ satisfies 
\[
 \int_{\Omega(1/2,2)}\frac{\dd A(z)}{|z|^2}
 =\pi\log\frac{b+2}{b+1/2}
       \sum_{j=1}^m\frac{1}{p_j^2}
 <\infty.
\]
\end{lemma}

\begin{proof}
Under $w=P_j(z)$ one has
$\dd A(z)/|z|^2=p_j^{-2}\dd A(w)/|w|^2$. Up to a boundary of area
zero, we have
\[
 \int_{T_j(1/2)\setminus T_j(2)}
       \frac{\dd A(z)}{|z|^2}
 =
 \frac1{p_j^2}\int_{b+1/2}^{b+2}
       \int_{-\infty}^{\infty}\frac{\dd y\dd x}{x^2+y^2}
 =\frac{\pi}{p_j^2}\int_{b+1/2}^{b+2}\frac{\dd x}{x}.
\]
Summing over the finitely many components proves the claim.
\end{proof}

\subsection{Variable width strip models}\label{sec:strips}
To treat the case \(\rho>\rho_E\), we consider conformal maps from strips of variable width onto a horizontal strip. As in the case \(\rho=\rho_E\), we need to verify the finite logarithmic area condition in \Cref{thm:bishop}. For this purpose, we first establish a uniform lower bound for the derivatives of these conformal maps using estimates for the hyperbolic density. This estimate will also be used to control the growth of the corresponding models.
\begin{lemma}\label[lemma]{lem:strip-derivative}
Let $0<a\leq b<\infty$ and $L,M\geq 0$ be constants.
Suppose that $h:\R\to(0,\infty)$ is twice continuously
differentiable and satisfies
\[
a\leq h(x)\leq b,\qquad
|h'(x)|\leq L,\qquad |h''(x)|\leq M \qquad (x\in\R).
\]Let
\(
U=\{x+iy:|y|<h(x)\}\) and
\(
S=\{w:|\operatorname{Im}w|<\pi/2\}.
\)
If $\chi:U\to S$ is the conformal isomorphism
normalized by $\chi(0)=0$ and preserving the two ends
in the sense that
\[
\lim_{\substack{z\in U\\ \operatorname{Re}z\to-\infty}}
\operatorname{Re}\chi(z)=-\infty,
\qquad
\lim_{\substack{z\in U\\ \operatorname{Re}z\to+\infty}}
\operatorname{Re}\chi(z)=+\infty,
\]
then
\[
\chi(\bar z)=\overline{\chi(z)}
\quad (z\in U),
\qquad
\chi'(x)>0
\quad (x\in\R).
\]
Moreover, there exists a constant $c>0$,
depending only on $a,b,L,M$, such that
\[
|\chi'(z)|\geq c,
\] for $z\in U$.
\end{lemma}

\begin{proof}
By reflection and uniqueness of the normalized map,
we have
\[
\chi(\bar z)=\overline{\chi(z)}.
\]
In particular, $\chi$ maps $\R$ onto $\R$.
Since $\chi$ preserves the two ends, it follows that
$\chi'(x)>0$ for all $x\in\R$, and hence
$\chi(U_+)=S_+$, where
$U_+:=U\cap\{\operatorname{Im} z>0\}$ and
$S_+:=S\cap\{\operatorname{Im} w>0\}$.

Write $d(z)=\operatorname{dist}(z,\partial U)$ and define
\[
 v(z)=\frac{\pi}{2}-\operatorname{Im}\chi(z).
\]
Then, $v$ is positive and harmonic on $U$, with
$v=\pi/2$ on $\R$ and $0<v\leq\pi/2$ in $U_+$. We first show that
\[
 v(z)\geq C d(z),
 \qquad z\in U_+
\]
for a constant $C>0$ depending only on $a,b,L,M$.

 At each point $P=(u,h(u))$ on the upper boundary of $U$,
let
\[
\nu=\frac{(-h'(u),1)}{\sqrt{1+h'(u)^2}}
\]
be the outward unit normal to $U$ at $P$.
Then, we claim that, for a sufficiently small $r>0$ depending only
on $a$ and $M$, the disk centered at $q=P-r\nu$
satisfies
\[
D(q,r)\subset U_+.
\]
To prove this, choose $r>0$ sufficiently small that
$2r<a$ and $rM<1$. For $P+\xi:=(u+\xi_1, h(u)+\xi_2)\in D(q,r)$, we have
$$|(P+\xi)-q|=|\xi+r\nu|<r,$$ and hence,
$$\nu\cdot\xi<-\frac{|\xi|^2}{2r}.$$ 
Let $F(x,y)=y-h(x)$. 
By Taylor's inequality and the bound on $h''$, we have
\[
\begin{aligned}
 F(P+\xi)
 &=h(u)+\xi_2-h(u+\xi_1)\\
 &\leq\xi_2-h'(u)\xi_1+\frac M2\xi_1^2\\
 &\leq \sqrt{1+h'(u)^2} \nu\cdot\xi+\frac M2|\xi|^2
 <-\frac12\left(\frac{\sqrt{1+h'(u)^2}}{r}-M\right)|\xi|^2<0.
\end{aligned}
\]
where the last inequality follows from
$\sqrt{1+h'(u)^2}\geq1$ and $rM<1$.
Moreover, since $|\xi|<2r$, we have
\[
\operatorname{Im}(P+\xi)
>h(u)-2r\geq a-2r>0.
\]
Thus $P+\xi\in U_+$, proving the claim. In particular, $q\in U_+$ and $d(q)\ge r$.

To estimate $v$ away from the boundary, we consider the following two cases according to
whether $d(z)\ge r$ or $d(z)<r$. 

\noindent\textbf{Case 1:} $z=x+iy\in U_+$ with $d(z)\ge r$.

Define
\(
 \Gamma_\pm:=\{u\pm ih(u):u\in\R\}.
\) 
For $0\leq t\leq y$ and $u\in\R$, the Lipschitz bound on $h$ and the
Cauchy--Schwarz inequality give
\[
\begin{aligned}
 h(x)-t
 &\leq |h(x)-h(u)|+|h(u)-t|\\
 &\leq L|x-u|+|h(u)-t|\\
 &\leq \sqrt{1+L^2}\times\sqrt{(x-u)^2+(h(u)-t)^2}.
\end{aligned}
\]
Taking the infimum over $u$ yields
\[
 \operatorname{dist}(x+it,\Gamma_+)
 \geq\frac{h(x)-t}{\sqrt{1+L^2}}.
\]
On the other hand, as $t\geq0$ and $h(u)\geq a$,
\[
\dist(x+it,\Gamma_-)=\inf_{u\in\R}\sqrt{(x-u)^2+(t+h(u))^2}\geq t+h(u)\geq a.
\]
Since $\partial U=\Gamma_+\cup\Gamma_-$,
we have
\[
d(x+it)\geq
\min\left\{a,\frac{h(x)-t}{\sqrt{1+L^2}}\right\}.
\]
Moreover, since $d(z)\geq r$ and
\[
h(x)-t\geq h(x)-y\geq d(z)\geq r,
\]
we obtain
\[
d(x+it)\geq
\min\left\{a,\frac{r}{\sqrt{1+L^2}}\right\}
=\frac{r}{\sqrt{1+L^2}},
\qquad 0\leq t\leq y.
\]

Let $\lambda_U$ denote the hyperbolic density of curvature $-1$ on
$U$. Since $U$ is simply connected, the holomorphic function
\(
H(z)=\frac{\pi}{2}+i\chi(z)
\)
maps $U$ into the right half-plane and satisfies
$\operatorname{Re}H=v$. Hence, by the Schwarz--Pick
lemma and the standard upper bound for the hyperbolic density (see
\cite[Theorems~6.4 and~8.2]{BeardonMinda}),  we have
\begin{equation*}
 |\nabla\log v(\zeta)|\leq\lambda_U(\zeta)
 \leq\frac2{d(\zeta)},\qquad\zeta\in U.
\end{equation*}
Integrating along the vertical segment, we obtain
\[
\begin{aligned}
 \log v(z)-\log v(x)
 &=\int_0^y\nabla\log v(x+it)\cdot(0,1)\,dt\\
 &\geq-\int_0^y|\nabla\log v(x+it)|\,dt\\
 &\geq-2\int_0^y\frac{dt}{d(x+it)}
 \geq-\frac{2y\sqrt{1+L^2}}{r}\geq-\frac{2b\sqrt{1+L^2}}{r}.
\end{aligned}
\]
Hence, for $z\in U_+$ with $d(z)\ge r$, 
\begin{equation*}
 v(z)\geq\frac\pi2\exp\left(-\frac{2b\sqrt{1+L^2}}{r}\right)=:m_0>0.
\end{equation*}

\noindent\textbf{Case 2:} $z=x+iy\in U_+$ with
$0<d(z)<r$.

Let $P$ be a nearest boundary point to $z$. Since
$\dist(z,\Gamma_-)\geq a>d(z)$, it follows that $P\in\Gamma_+$ and 
\(
z=P-d(z)\nu,
\)
where $\nu$ is the outward unit normal at $P$.
For $q=P-r\nu$ as above, we have
\[
|z-q|=r-d(z)<r.
\]
Since $d(q)\geq r$, from Case 1, we have $v(q)\geq m_0$.
Applying Harnack's inequality in $D(q,r)$,
we obtain
\[
\begin{aligned}
v(z)
&\ge \frac{r-|z-q|}{r+|z-q|}v(q)=\frac{d(z)}{2r-d(z)}v(q)
\ge \frac{m_0}{2r}d(z).
\end{aligned}
\]
Thus, for $z\in U_+$ with $d(z)<r$, 
\begin{equation*}
v(z)\geq\frac{m_0}{2r}d(z).
\end{equation*}
Finally, since $d(z)\leq h(x)-y\leq b$
for $z=x+iy\in U_+$,
and $2r<a\leq b$, combining
both two cases yields
\begin{equation}\label{eq:strip-v-distance}
v(z)\geq\frac{m_0}{b}d(z),\quad z\in U_+.
\end{equation}

By conformal invariance of the hyperbolic metric,
\[
|\chi'(z)|=\frac{\lambda_U(z)}{\lambda_S(\chi(z))}
=\lambda_U(z)\cos(\operatorname{Im}\chi(z))
=\lambda_U(z)\sin v(z).
\]
Since
\(
\lambda_U(z)\ge\frac{1}{2d(z)}
\) (see \cite[Theorem~8.6]{BeardonMinda})
and $\sin t\ge 2t/\pi$ for $0\le t\le\pi/2$,
we obtain
\[
|\chi'(z)|
\ge\frac{1}{2d(z)}\frac{2v(z)}{\pi}
\ge\frac{c}{\pi},
\qquad z\in U_+,
\] for some constant $c>0$ depending only on $a, b, L, M$.

By symmetry, the same estimate holds on
the lower half of $U$, completing the proof.
\end{proof}

\begin{lemma}\label[lemma]{lem:strip-area-growth}
Assume the hypotheses of \Cref{lem:strip-derivative}, and suppose in
addition that $h(x)<\pi$ for every $x$. Let
\(
 \Omega_\theta=e^{i\theta}\exp U,\ \text{for}\ \theta\in\R\ \text{and}\ 
 \tau(z)=\exp\chi(\operatorname{Log}(e^{-i\theta}z)),\ \text{for}\ z\in\Omega_{\theta},
\)
where $\operatorname{Log}$ denotes the principal branch with $\operatorname{Log}(e^{-i\theta}\Omega_\theta)=U$. Then $\tau:\Omega_\theta
\to\Hh$ is conformal, and for $0<p<q$,
\begin{equation*}
 \int_{\{z\in\Omega_\theta:p<\operatorname{Re}\tau(z)<q\}}
       \frac{\dd A(z)}{|z|^2}
 \leq\frac{\pi}{c^2}\log\frac qp.
\end{equation*}where $c>0$ is the constant from
\Cref{lem:strip-derivative}.
Moreover, 
\begin{equation}\label{eq:strip-vertical-growth}
 e^{\chi(x)}
 \leq\sup_{|y|<h(x)}\operatorname{Re}e^{\chi(x+iy)}
 \leq\exp\left(\chi(x)+\frac{\pi^2}{4ac}\right), \quad x\in\R
\end{equation}
and $|\tau(z_n)|\to\infty$ implies $|z_n|\to\infty$, for any sequence $(z_n)\subset\Omega_\theta$.
\end{lemma}

\begin{proof}
Since $U\subset\{z:|\operatorname{Im}z|<\pi\}$, the exponential is
injective on $U$. Also $\exp:S\to\Hh$ is conformal. This proves the
first assertion. 

Let $\psi=\chi^{-1}$, by
\Cref{lem:strip-derivative}, $|\psi'|\leq c^{-1}$ on $S$.
The changes of variables $z=e^{i\theta}e^\zeta$ and
$w=\chi(\zeta)=s+it$ give
\[
 \int_{\{z\in\Omega_\theta:p<\operatorname{Re}\tau(z)<q\}}
       \frac{\dd A(z)}{|z|^2}=\int_{W_{p,q}}|\psi'(w)|^2\dd A(w)\le c^{-2}\operatorname{area}(W_{p,q})=c^{-2}\pi\log\frac{q}{p},\]
 where \(
 W_{p,q}=\{s+it:|t|<\pi/2,\ p<e^s\cos t<q\}
\).

Next, by conformal invariance of the hyperbolic metric
and the inclusions
\[
\{z:|\operatorname{Im}z|<a\}
\subset U\subset
\{z:|\operatorname{Im}z|<b\},
\] we have
\begin{equation*}
 \chi'(x)=\lambda_U(x),\ \text{and}\
 \frac{\pi}{2b}\leq \chi'(x)\leq\frac{\pi}{2a}.
\end{equation*}
If $\chi(x+iy)=s+it$, integration of $\psi'$ along the vertical
segment from $s$ to $s+it$ gives
\[
 |x-\psi(s)|\leq|\psi(s+it)-\psi(s)|\leq\frac{\pi}{2c}.
\]
Since $\chi(\psi(s))=s$, it follows that
$$|s-\chi(x)|=|\chi(\psi(s))-\chi(x)|\leq\frac{\pi^2}{4ac}.$$ This proves the upper bound in
\eqref{eq:strip-vertical-growth}. The lower bound follows by taking
$y=0$. 

Finally, suppose that $\{z_n\}\subset\Omega_\theta$
satisfies $|\tau(z_n)|\to\infty$.
Since
\[
\chi\bigl(\text{Log}(e^{-i\theta}z_n)\bigr)
=s_n+it_n,
\]
we have $s_n=\log|\tau(z_n)|\to\infty$.
Hence,
\[
\log|z_n|
\geq\psi(s_n)-\frac{\pi}{2c}
\longrightarrow\infty,
\]
since $\psi(s)\to\infty$ as $s\to\infty$
by preservation of the two ends.
This completes the proof.
\end{proof}

\section{The inverse problem on class \texorpdfstring{$\Bcal$}{B} for finite orders}\label{sec:realB}

We now prove \Cref{thm:arcs-main}. The power sector models give the case
$\rho=\rho_E$. For $\rho>\rho_E$, we construct strips whose widths are
constant except on a sequence of widely separated intervals. We then
complete both constructions by the same scaling argument.

\subsection{The case \texorpdfstring{$\rho=\rho_E$}{rho = rho E}}\label{sec:rho-equal}
Recall the power sector model constructed in
Section~\ref{sec:B}. By \Cref{lem:logarea},
the hypotheses of \Cref{thm:bishop} are satisfied.
Applying that theorem, we obtain
\[
Q=f_0\circ\phi.
\]

We first estimate the growth of $Q$.
On each $T_j(2)$, we have
\[
Q(z)=\exp(P_j(z)-b),
\]
while $|Q(z)|\leq e^2$ outside $\Omega(2)$.
Choose an index $j$ such that $p_j=\rho_E$.
For sufficiently large $r$, we have
\[re^{i\theta_j}\in T_j(2),\quad \text{and hence}\quad
\log|Q(re^{i\theta_j})|
=r^{\rho_E}-b.
\]
This gives
\[
\log M(r,Q)\geq r^{\rho_E}-b.
\]
On the other hand, for $z\in T_j(2)$ with
$|z|=r\geq1$, we have
\[
\log|Q(z)|
=\operatorname{Re}P_j(z)-b
\leq r^{p_j}
\leq r^{\rho_E}.
\]
Together with the bound outside $\Omega(2)$,
this yields
\[
r^{\rho_E}-b
\leq\log M(r,Q)
\leq\max\{2,r^{\rho_E}\}
\]
for all sufficiently large $r$. Consequently,
\[
\lim_{r\to\infty}
\frac{\log\log M(r,Q)}{\log r}
=\rho_E.
\]

The proof of \Cref{lem:sector-direction} shows that,
for every $j$ and every
$\theta=\theta_j+\delta$ with $|\delta|<b_j$,
the ray $z_r=re^{i\theta}$ eventually lies in
$T_j(2)$, and
\[
\operatorname{Re}\tau(z_r)
=\operatorname{Re}P_j(z_r)-b
\longrightarrow\infty
\qquad(r\to\infty).
\]

\subsection{The case \texorpdfstring{$\rho>\rho_E$}{rho > rho E}}
We use the notation for the arcs
from Section~\ref{sec:power-models}. 
Let
\(
a=\pi/(2\rho).
\)
Then $$0<a<b_j\leq\pi$$ for every $j$, with
$b_j=\pi$ possible only when $m=1$ and $E=\T$.
To apply \Cref{lem:strip-derivative,lem:strip-area-growth}, we construct smooth variable width strips with uniformly bounded first and second derivatives of their half width functions. We arrange that these functions equal \(a\) outside a sequence of widely separated intervals, while approaching \(b_j\) along a sequence tending to infinity. 

Choose $\beta\in C_c^\infty((-1,1))$ such that
\[
0\leq\beta(x)\leq1 \quad (x\in\R),\quad\text{and}
\quad
\beta(x)=1 \quad (|x|\leq1/2).
\]
For $n\geq2$, let
\[
x_n=2^{n^2},\qquad d_n=n^2,\qquad
b_{j,n}=a+(b_j-a)\frac{n}{n+1},
\]
and define
\[
\beta_n:=\beta\left(\frac{x-x_n}{d_n}\right)\quad \text{and}\quad 
h_j(x):=a+\sum_{n\geq2}(b_{j,n}-a)\beta_n.
\]
Since the intervals $(x_n-d_n,x_n+d_n)$
are pairwise disjoint, each $h_j$ is smooth
and satisfies
\[
a\leq h_j(x)<b_j\leq\pi,\qquad
h_j(x_n)=b_{j,n}\longrightarrow b_j.
\]
Moreover, $h_j(x)=a$ for $x\leq0$. Since the supports
\[
\operatorname{supp}\left(\beta_n\right)
\subset (x_n-d_n,x_n+d_n)
\]
are pairwise disjoint and $d_n\geq4$, we have
\[
\|h_j'\|_\infty\leq\sup_n\frac{b_{j,n}-a}{d_n}\|\beta'\|_{\infty}
\leq\frac{\pi-a}{4}\|\beta'\|_\infty,
\] and
\[
\|h_j''\|_\infty\le
\sup_n\frac{b_{j,n}-a}{d_n^2}
\|\beta''\|_{\infty}
\leq\frac{\pi-a}{16}\|\beta''\|_\infty.
\]
Thus the first and second derivatives of $h_j$
are uniformly bounded, independently of $j$.

Let
\(
U_j=\{x+iy:|y|<h_j(x)\},
\)
and let $\chi_j:U_j\to S$ be the conformal
map normalized as in \Cref{lem:strip-derivative},
where $S=\{w:|\operatorname{Im}w|<\pi/2\}$.
By \Cref{lem:strip-derivative}, there exists
a constant $c>0$, independent of $j$, such that
\[
|\chi_j'(z)|\geq c,\qquad z\in U_j.
\]

For the growth estimates, we use the following
symmetric form of Warschawski's strip asymptotic
theorem \cite[Theorem~IX, p.~313]{Warschawski}.

\begin{lemma}[Warschawski]
\label[lemma]{lem:warschawski-strip}
Assume the hypotheses of \Cref{lem:strip-derivative} and suppose in addition that
\[
 h'(x)\longrightarrow0\quad(x\to+\infty),\qquad
 \int_0^\infty |h''(t)|\dd t<\infty,\qquad
 \int_0^\infty\frac{h'(t)^2}{h(t)}\dd t<\infty.
\]
Then there is a constant $\kappa\in\R$ such that
\[
 \lim_{x\to+\infty}\sup_{|y|<h(x)}
 \left|\chi(x+iy)-\frac{\pi}{2}\int_0^x\frac{\dd t}{h(t)}
       -\kappa-i\frac{\pi y}{2h(x)}\right|=0.
\]
\end{lemma}
In Warschawski's notation, the boundary functions
are $\varphi_\pm=\pm h$, so the width is $2h$
and the midpoint is $0$. The condition $h'(x)\to0$
implies that $U$ is an $L$-strip with boundary
inclination zero, while
$\int_0^\infty |h''(t)|\dd t<\infty$
ensures finite boundary turning in the sense
of Warschawski. Moreover,
\[
\int_0^\infty
\frac{((2h)'(t))^2}{2h(t)}\dd t
=
2\int_0^\infty\frac{h'(t)^2}{h(t)}\dd t<\infty.
\]
Thus the hypotheses of
\cite[Theorem~VII, pp.~310--311]{Warschawski}
are satisfied, and the conclusion follows from
\cite[Theorem~IX, p.~313]{Warschawski}.

\begin{lemma}\label[lemma]{lem:width-growth}
For every $j$,
\[
 \frac{\chi_j(x)}x\longrightarrow\rho\qquad(x\to+\infty).
\]
\end{lemma}

\begin{proof}
We first verify the hypotheses of \Cref{lem:warschawski-strip}.
 On the $n$th support interval of $\beta_n$,
\[
 |h_j'(x)|\leq\frac{b_{j,n}-a}{d_n}\|\beta'\|_\infty
 \leq\frac{b_j-a}{n^2}\|\beta'\|_\infty.
\]
Since $h_j'=0$ outside \(\bigcup_{n\ge2}[x_n-d_n,x_n+d_n],
\) it follows
that $h_j'(x)\to0$ as $x\to+\infty$.
Moreover, the supports are pairwise disjoint,
and a change of variables gives
\[
\begin{aligned}
\int_0^\infty |h_j''(x)|\,dx
&=\sum_{n\ge2}
\frac{b_{j,n}-a}{d_n^2}
\int_{x_n-d_n}^{x_n+d_n}
\left|
\beta''\left(\frac{x-x_n}{d_n}\right)
\right|dx\\
&=\sum_{n\geq2}\frac{b_{j,n}-a}{d_n}\int_{-1}^1|\beta''(t)|dt\\
&
\le (b_j-a)\sum_{n\ge2}\frac1{n^2}\int_{-1}^1|\beta''(t)|dt
<\infty.
\end{aligned}
\]
and
\[
\begin{aligned}
\int_0^\infty\frac{h_j'(x)^2}{h_j(x)}\,dx
&\le\frac1a\int_0^\infty |h_j'(x)|^2\,dx=\frac{1}{a}\sum_{n\ge2}\frac{(b_{j,n}-a)^2}{d_n^2}
\int_{x_n-d_n}^{x_n+d_n}
\left|
\beta'\left(\frac{x-x_n}{d_n}\right)
\right|^2dx\\
&=
\frac 1a\sum_{n\ge2}\frac{(b_{j,n}-a)^2}{d_n}\int_{-1}^1|\beta'(t)|^2\,dt\le\frac{(b_j-a)^2}{a}\sum_{n\ge2}\frac1{n^2}\int_{-1}^1|\beta'(t)|^2\,dt<\infty.
\end{aligned}
\]

By \Cref{lem:warschawski-strip}, there exists
$\kappa_j\in\R$ such that
\begin{equation}\label{eq:warschawski-expansion}
\chi_j(x+iy)
=
\frac{\pi}{2}\int_0^x h_j(t)^{-1}\dd t+\kappa_j
+i\frac{\pi y}{2h_j(x)}
+\varepsilon_j(x,y),
\end{equation}
where
\begin{equation}\label{eq:warschawski-uniform-error}
E_j(x):=
\sup_{|y|<h_j(x)}|\varepsilon_j(x,y)|
\longrightarrow0
\qquad(x\to+\infty).
\end{equation}

It remains to show that $$\displaystyle\frac{\pi}{2x}\int_0^x h_j(t)^{-1}\dd t\to\rho.$$
Let
\[
B=\bigcup_{n\geq2}[x_n-d_n,x_n+d_n].
\]
For sufficiently large $X$, let $N$ be the
largest integer such that $x_N-d_N\leq X$.
Then
\[
\operatorname{meas}(B\cap[0,X])
\leq 2\sum_{n=2}^N n^2
=O(N^3).
\]
Since
\[
X\geq x_N-d_N=2^{N^2}-N^2,
\]
we obtain
\begin{equation*}
\operatorname{meas}(B\cap[0,X])=o(X).
\end{equation*}

Since $h_j=a$ outside $B$ and $h_j\geq a$,
and recalling that $\rho=\pi/(2a)$, we have
\[
\begin{aligned}
0\leq \rho x-\frac{\pi}{2}\int_0^x h_j(t)^{-1}\dd t
&=\frac{\pi}{2}
\int_{B\cap[0,x]}
\left(\frac1a-\frac1{h_j(t)}\right)\dd t\\
&\leq
\rho\,\operatorname{meas}(B\cap[0,x])
=o(x).
\end{aligned}
\]
Thus, $\displaystyle\frac{\pi}{2x}\int_0^x h_j(t)^{-1}\dd t\to\rho$.
Taking $y=0$ in \eqref{eq:warschawski-expansion}
gives the desired conclusion.
\end{proof}
In particular, \eqref{eq:warschawski-expansion}
and \eqref{eq:warschawski-uniform-error} imply that
\begin{equation}\label{eq:strip-uniform-asymptotic}
\sup_{|y|<h_j(x)}
\left|
\operatorname{Re}\chi_j(x+iy)-\chi_j(x)
\right|
\longrightarrow0
\qquad(x\to+\infty).
\end{equation}

We now use the variable width strips constructed above to build a model for the case \(\rho>\rho_E\). 
For each $j$, let
\[
\Omega_j=e^{i\theta_j}\exp U_j,
\qquad
\tau_j(z)=\exp\chi_j\bigl(
\operatorname{Log}(e^{-i\theta_j}z)\bigr),
\]
and define
\[
\Omega=\bigcup_{j=1}^m\Omega_j,
\qquad
\tau|_{\Omega_j}=\tau_j.
\]
Since $h_j(x)<b_j=\ell_j/2$, each $\Omega_j$
is contained in the sector with central direction
$\theta_j$ and opening $\ell_j$.
Hence the domains $\Omega_j$ are pairwise disjoint.
Moreover, the boundary of each $\Omega_j$ consists
of two curves meeting at the origin and is therefore
connected in $\C$.

By \Cref{lem:strip-area-growth}, each
$\tau_j:\Omega_j\to\Hh$ is conformal and satisfies
the required condition at infinity.
Thus $(\Omega,e^\tau)$ defines a model.
Furthermore, the same lemma gives
\[
\int_{\Omega(1/2,2)}
\frac{\dd A(z)}{|z|^2}
\leq\frac{m\pi}{c^2}\log4<\infty.
\]
Hence \Cref{thm:bishop} applies, yielding
$Q=f_0\circ\phi$.

\begin{lemma}\label[lemma]{lem:variable-model-growth}
Let
\[
A(x)=\max_{1\leq j\leq m}\chi_j(x).
\]
Then
\begin{equation}\label{eq:variable-model-sharp-growth}
\log\log M(e^x,Q)=A(x)+o(1)
\qquad(x\to+\infty).
\end{equation}
In particular,
\[
\lim_{r\to\infty}
\frac{\log\log M(r,Q)}{\log r}=\rho.
\]
Moreover, for every $j$ and every
$\theta=\theta_j+y$ with $|y|<b_j$,
the sequence $z_n=e^{x_n+i\theta}$ eventually
lies in $\Omega_j$, with $|z_n|\to\infty$
and $\arg z_n=\theta\pmod{2\pi}$, and satisfies
\begin{equation*}
\operatorname{Re}\tau_j(z_n)
=e^{\chi_j(x_n)}
\left(\cos\left(\frac{\pi y}{2b_j}\right)+o(1)\right)
\longrightarrow\infty.
\end{equation*}
\end{lemma}

\begin{proof}
By \Cref{lem:width-growth}, we have
$A(x)/x\to\rho$ as $x\to+\infty$.
Since $\chi_j(x)\to\infty$ for every $j$,
the points $e^{x+i\theta_j}$ belong to
$\Omega_j(2)$ for all sufficiently large $x$.
Thus \eqref{eq:model-agreement} gives
\[
\log M(e^x,Q)
\geq\max_{1\leq j\leq m}
\operatorname{Re}\tau_j(e^{x+i\theta_j})
=e^{A(x)}.
\]
To obtain the upper bound, fix $\eta>0$.
By \eqref{eq:strip-uniform-asymptotic},
there exists $T>0$ such that
\[
\operatorname{Re}\chi_j(s+iy)
\leq\chi_j(s)+\eta
\]
for all $j\ge1$, $s\geq T$, and $|y|<h_j(s)$.

For $z\in\Omega_j(2)$ with $|z|\leq e^x$,
write
\[
\operatorname{Log}(e^{-i\theta_j}z)=s+iy,
\qquad s\leq x.
\]
If $s\geq T$, then the monotonicity of
$\chi_j$ on $\R$ gives
\[
\begin{aligned}
\log|Q(z)|
&=\operatorname{Re}\tau_j(z)\leq e^{\operatorname{Re}\chi_j(s+iy)}
\leq e^{\chi_j(s)+\eta}
\leq e^{A(x)+\eta}.
\end{aligned}
\]
If $s<T$, then \eqref{eq:strip-vertical-growth}
gives
\[
\log|Q(z)|
=\operatorname{Re}\tau_j(z)
\leq e^{\chi_j(s)+K}
\leq e^{A(T)+K},
\qquad K=\frac{\pi^2}{4ac}.
\]
Since $|Q|\leq e^2$ outside $\Omega(2)$
by \eqref{eq:model-agreement}, we obtain
\[
e^{A(x)}
\leq\log M(e^x,Q)
\leq\max\{2,e^{A(T)+K},e^{A(x)+\eta}\}.
\]
As $A(x)\to\infty$, it follows that
\[
0\leq
\log\log M(e^x,Q)-A(x)
\leq\eta
\]
for all sufficiently large $x$.
Since $\eta>0$ is arbitrary,
\eqref{eq:variable-model-sharp-growth} follows.
Together with $A(x)/x\to\rho$,
this gives the asserted order limit.

Finally, fix $j$ and
$\theta=\theta_j+y$ with $|y|<b_j$.
Since $h_j(x_n)=b_{j,n}\to b_j$,
we have $x_n+iy\in U_j$ for all
sufficiently large $n$, and hence
\[
z_n=e^{x_n+i\theta}\in\Omega_j.
\]
By \eqref{eq:warschawski-expansion}
and \eqref{eq:warschawski-uniform-error},
\[
\begin{aligned}
\chi_j(x_n+iy)-\chi_j(x_n)
&=\frac{\pi i y}{2h_j(x_n)}+o(1)\longrightarrow\frac{\pi i y}{2b_j}.
\end{aligned}
\]
Therefore,
\[
\begin{aligned}
\frac{\tau_j(z_n)}{e^{\chi_j(x_n)}}
&=\exp\bigl(
\chi_j(x_n+iy)-\chi_j(x_n)
\bigr)\longrightarrow
\exp\left(\frac{\pi i y}{2b_j}\right).
\end{aligned}
\]
Taking real parts yields
\[
\operatorname{Re}\tau_j(z_n)
=e^{\chi_j(x_n)}
\left(
\cos\left(\frac{\pi y}{2b_j}\right)+o(1)
\right).
\]
Since $|y|<b_j$ and
$\chi_j(x_n)\sim\rho x_n\to\infty$,
it follows that $\operatorname{Re}\tau_j(z_n)\to\infty$. This completes the proof.
\end{proof}

\subsection{Completion of the proof}

In both constructions, the model has $m$ components $\Omega_j$,
and every point of $\Omega_j$ has its argument in $I_j$. By Section~\ref{sec:rho-equal} and \Cref{lem:variable-model-growth}, for each direction
$\theta=\theta_j+y$ with $|y|<b_j$, there are $z_n\in\Omega_j$
with
\begin{equation}\label{eq:model-direction-sequences}
 |z_n|\to\infty,\qquad\arg z_n\to\theta,\qquad
 \operatorname{Re}\tau_j(z_n)\to\infty.
\end{equation}
Closedness therefore gives
\begin{equation}\label{eq:all-model-directions}
 L(\Omega_j(s))=I_j\quad(s>0),\qquad L(\Omega(s))=E.
\end{equation}
When $E=\T$, the directions with $|y|<\pi$ are dense in the circle,
so this argument includes the opposite direction as well.

In either case we have $Q=f_0\circ\phi$, $\phi(z)=z+o(|z|)$, and
\[
 \frac{\log\log M(r,Q)}{\log r}\longrightarrow\rho.
\]
For all sufficiently large $r$,
\[
 \phi(D(0,r/2))\subset D(0,r)\subset\phi(D(0,2r)).
\]
It is not hard to see that
\[
  M(r/2,Q)\leq M(r,f_0)\leq M(2r,Q),
\]
and hence, $\lambda(f_0)=\rho(f_0)=\rho$.

Choose $\eta>0$ so small that
\begin{equation*}
 \eta e^2<\frac12,\qquad
 \eta M(1,f_0)<\frac12,
\end{equation*}
and put $f=\eta f_0$. As in the proof of \Cref{thm:compact-main}, \( 
 f(\overline\D)\subset D(0,1/2)\subset\D,
\)
 and $f$ satisfies
$\lambda(f)=\rho(f)=\rho$.

If $w\notin\phi(\Omega(2))$, then $|f(w)|\leq\eta e^2<1/2$,
so
\(
\J(f)\subset\phi(\Omega(2)).
\)
By \eqref{eq:phi-directions} and
\eqref{eq:all-model-directions}, we also have
\[
\Lcal(f)\subset\Lcal(\phi(\Omega(2)))
=\Lcal(\Omega(2))=E.
\]

Conversely, let $\theta$ and $z_n$ be as in
\eqref{eq:model-direction-sequences}, and let $w_n=\phi(z_n)$.
For large $n$, 
\[z_n\in\Omega_j(2)\quad \text{and}\quad
|f(w_n)|=\eta\exp(\operatorname{Re}\tau_j(z_n))\longrightarrow\infty.
\]
Hence, $|w_n|\to\infty$ and $\arg w_n\to\theta$. By
\Cref{prop:large-values-julia}, $\theta\in\Lcal(f)$.
Closedness gives the endpoints of all the arcs, so $\Lcal(f)=E$.
By \Cref{prop:direction-equality-B}, we also have
$\Lcal_A(f)=\Lcal_I(f)=E$.
This completes the proof of \Cref{thm:arcs-main}.

\section{The meromorphic counterpart}\label{sec:mero}

We now prove \Cref{thm:meromorphic-main}. The order and direction estimates
come from the construction of the last two authors. The new ingredient is the
bounded singular value argument, which we give in detail.

Let $E_0$ be a countable dense subset of $E$, which may be finite. Choose a sequence
$\{\widehat\theta_k\}_{k\geq1}$ in $E_0$ such that
every point of $E_0$ appears infinitely many times
in the sequence.

Following \cite[Section~3]{WangYao}, put
\begin{equation*}
 m_1=5,\qquad m_{k+1}=2^{m_1m_2\cdots m_k},
\end{equation*}
and choose
\begin{equation*}
 r_k=
 \begin{cases}
  m_k^{1/\rho},&0<\rho<\infty,\\
  m_k^k,&\rho=0,\\
  m_k^{1/k},&\rho=\infty.
 \end{cases}
\end{equation*}
Clearly, $r_k\to\infty$ and
$r_{k+1}/r_k\to\infty$.
By discarding finitely many initial terms and
relabeling, we may assume that
\[
r_{k+1}-r_k> 4,
\qquad m_k\geq k+4
\qquad(k\geq1).
\]

Now, consider the series 
$$G(z)=\sum_{k=1}^{\infty}(z-a_k)^{-m_k},\quad a_k=r_ke^{i\widehat\theta_k}.$$
Since $|a_k|=r_k\to\infty$, for every compact
set $K\subset\C$, we have
\[
\inf_{z\in K}|z-a_k|
\geq r_k-\sup_{z\in K}|z|
\longrightarrow\infty,
\]
and for sufficiently large $k$, \[
\sum_{n\geq k}|z-a_n|^{-m_n}
\leq \sum_{n\geq1}2^{-m_n}<\infty.
\] uniformly on $K$.
Hence, $G$ is a transcendental meromorphic function
with poles $a_k$ of multiplicities $m_k$.

\begin{lemma}\label[lemma]{lem:pole-separation}
Let $h_k(z)=G(z)-(z-a_k)^{-m_k}$, $k\in\mathbb{N}$.
Then there exist constants $B_0,B_1,H>0$,
independent of $k$, such that
\begin{equation*}
\sup_{z\in D(a_k,1)}|h_k(z)|\leq B_0,
\qquad
\sup_{z\in D(a_k,1)}|h_k'(z)|\leq B_1,
\end{equation*}
and
\begin{equation}\label{eq:outside-pole-disks}
|G(z)|\leq H,
\qquad
z\notin\bigcup_{k\geq1}D(a_k,1).
\end{equation}
Moreover, $G'$ has no zeros in
$D(a_k,1)\setminus\{a_k\}$ for any $k$.
\end{lemma}

\begin{proof}
Since $r_{k+1}-r_k>4$, we have
\begin{equation}\label{eq:centre-separation}
|a_j-a_k|
\geq |r_j-r_k|
\geq 4|j-k|,
\qquad j\neq k,
\end{equation} and
$$D(a_k, 2)\cap D(a_j, 2)=\emptyset,\qquad  j\neq k$$

If $z\in D(a_k,1)$ and $j\neq k$, then
\[
|z-a_j|\geq |a_j-a_k|-|z-a_k|
> 4|j-k|-1
\ge 3|j-k|.
\]
Hence
\[
|h_k(z)|
\leq\sum_{j\neq k}(3|j-k|)^{-m_j}
\leq\sum_{j\geq1}3^{-m_j}
=:B_0
\]
and
\[
|h_k'(z)|
\leq\sum_{j\neq k}
m_j(3|j-k|)^{-m_j-1}
\leq\sum_{j\geq1}m_j3^{-m_j-1}
=:B_1.
\]
Both series converge since $m_j\geq j+4$.

If $z\not\in D(a_j ,2)$ for all $j$, then
\[
|G(z)|\leq\sum_{j\geq1}2^{-m_j}.
\]

If $z\in D(a_k,2)\setminus D(a_k,1)$, 
then, for $j\neq k$,
\[
|z-a_j|
> 4|j-k|-2
\ge 2|j-k|.
\]
Thus
\[
\sum_{j\neq k}|z-a_j|^{-m_j}
\leq\sum_{j\neq k}(2|j-k|)^{-m_j}
\leq\sum_{j\geq1}2^{-m_j}.
\]
Hence \eqref{eq:outside-pole-disks} holds with
\[
H:=1+\sum_{j\geq1}2^{-m_j}.
\]

Finally, since $m_j\geq j+4$, we have
\[
B_1
=\sum_{j\geq1}m_j3^{-m_j-1}
\leq\sum_{j\geq1}(j+4)3^{-j-5}
<1.
\]
For $0<|z-a_k|<1$, it follows that
\[
\begin{aligned}
|G'(z)|
&\geq m_k|z-a_k|^{-m_k-1}
     -|h_k'(z)|\\
&>m_k-B_1>0.
\end{aligned}
\]
This completes the proof.
\end{proof}

\begin{lemma}\label[lemma]{lem:meromorphic-B}
$G\in \Bmer$ and $S_{\mathrm{fin}}(G)
                         \subset\overline{D(0,H)}.$
Moreover, $\infty$ is not an asymptotic value of $G$.
\end{lemma}

\begin{proof}
By \Cref{lem:pole-separation}, every critical
point of $G$ with a finite critical value lies
outside $\bigcup_k D(a_k,1)$.
Hence all finite critical values have modulus
at most $H$.

Suppose that $G$ has an asymptotic value
$v\in\widehat{\C}$ with either $v=\infty$ or $|v|>H$.
Then there exists a curve $\gamma:[0,\infty)\to\C$
such that
\[
|\gamma(t)|\longrightarrow\infty,
\qquad
G(\gamma(t))\longrightarrow v.
\]
In either case, $|G(\gamma(t))|>H$ for all
sufficiently large $t$. By
\eqref{eq:outside-pole-disks}, the tail of $\gamma$
must lie in $\bigcup_k D(a_k,1)$.
Since these disks are pairwise disjoint and the
tail is connected, it must lie in one fixed disk
$D(a_k,1)$, contradicting $|\gamma(t)|\to\infty$.

Thus $\infty$ is not an asymptotic value, and
every finite asymptotic value has modulus at most
$H$. Together with the bound on critical values,
this gives
\(
S_{\mathrm{fin}}(G)\subset\overline{D(0,H)},
\)
and hence $G\in\Bmer$.

\end{proof}

\begin{proof}[Proof of \Cref{thm:meromorphic-main}]
By \cite[Proof of Theorem~1]{WangYao}, the original
series has Nevanlinna order $\rho$. Since deleting
finitely many terms does not change the order, we also have $\rho(G)=\rho$.

Choose $0<\lambda\leq1$ such that
\[
\lambda M(1,G)<\frac12,
\qquad
\lambda H<\frac12,
\]
and put $g:=\lambda G$. Then, $\rho(g)=\rho$.
By \Cref{lem:meromorphic-B} and the same argument
as in the proof of \Cref{thm:compact-main}, 
$g(\overline\D)\subset D(0,1/2)$ and $g\in\Bmer$.
 
Moreover, 
if
$z\notin\bigcup_{k\geq1}D(a_k,2)$, then
\[
|g(z)|
\leq\lambda\sum_{k\geq1}2^{-m_k}
<\frac12.
\]
Hence
\[
\J(g)\subset\bigcup_{k\geq1}D(a_k,2),
\]
which gives $\Lcal(g)\subset E$.
Since every pole $a_k$ belongs to $\J(g)$
and $\{\widehat\theta_k\}$ is dense in $E$,
we obtain $\Lcal(g)=E$. This completes the proof of \Cref{thm:meromorphic-main}.
\end{proof}

  \bibliographystyle{amsalpha}
  \bibliography{bibfile}

@article{Bergweiler,
  author  = {Bergweiler, W.},
  title   = {Iteration of meromorphic functions},
  journal = {Bull. Amer. Math. Soc. (N.S.)},
  volume  = {29},
  year    = {1993},
  pages   = {151--188}
}

@article{BergweilerKotus,
  author  = {Bergweiler, W. and Kotus, J.},
  title   = {On the {Hausdorff} dimension of the escaping set of certain meromorphic functions},
  journal = {Trans. Amer. Math. Soc.},
  volume  = {364},
  year    = {2012},
  pages   = {5369--5394},
  doi     = {10.1090/S0002-9947-2012-05587-2},
  url     = {https://doi.org/10.1090/S0002-9947-2012-05587-2}
}

@article{BishopB,
  author  = {Bishop, C. J.},
  title   = {Models for the {Eremenko--Lyubich} class},
  journal = {J. Lond. Math. Soc. (2)},
  volume  = {92},
  year    = {2015},
  pages   = {202--221},
  doi     = {10.1112/jlms/jdv021},
  url     = {https://doi.org/10.1112/jlms/jdv021}
}

@article{BER2024,
    author = {Bergweiler, W. and Eremenko, A. and Rempe, L.},
    title = {Second order linear differential equations with a basis of
solutions having only real zeros},
    journal = {J. Anal. Math.},
    year = {2024},
    volume = {152},
    pages = {53--108}
}

@incollection{BeardonMinda,
  author    = {Beardon, A. F. and Minda, D.},
  title     = {The hyperbolic metric and geometric function theory},
  booktitle = {Quasiconformal Mappings and Their Applications},
  editor    = {Ponnusamy, S. and Sugawa, T. and Vuorinen, M.},
  pages     = {9--56},
  year      = {2007},
  publisher = {Narosa Publishing House},
  address   = {New Delhi}
}

@article{Warschawski,
  title={On conformal mapping of infinite strips},
  author={Warschawski, S. E.},
  journal={Transactions of the American Mathematical Society},
  volume={51},
  pages={280--335},
  year={1942}
}

@article{BishopS,
  author  = {Bishop, C. J.},
  title   = {Models for the {Speiser} class},
  journal = {Proc. Lond. Math. Soc. (3)},
  volume  = {114},
  year    = {2017},
  pages   = {765--797},
  doi     = {10.1112/plms.12025},
  url     = {https://doi.org/10.1112/plms.12025}
}

@article{EL,
  author  = {Eremenko, A. E. and Lyubich, M. Y.},
  title   = {Dynamical properties of some classes of entire functions},
  journal = {Ann. Inst. Fourier (Grenoble)},
  volume  = {42},
  year    = {1992},
  pages   = {989--1020},
  doi     = {10.5802/aif.1318},
  url     = {https://doi.org/10.5802/aif.1318}
}

@article{Fletcher2021,
  author  = {Fletcher, A. N.},
  title   = {On {Julia} limiting directions in higher dimensions},
  journal = {Comput. Methods Funct. Theory},
  volume  = {21},
  year    = {2021},
  pages   = {587--603},
  doi     = {10.1007/s40315-021-00381-w},
  url     = {https://doi.org/10.1007/s40315-021-00381-w}
}

@article{FletcherSteranka2025,
  author  = {Fletcher, A. N. and Steranka, J. M.},
  title   = {{Phragm{\'e}n--Lindel{\"o}f} principles and {Julia} limiting directions of quasiregular mappings},
  journal = {Ergodic Theory Dynam. Systems},
  volume  = {45},
  year    = {2025},
  pages   = {3325--3343},
  doi     = {10.1017/etds.2025.10140},
  url     = {https://doi.org/10.1017/etds.2025.10140}
}

@book{Goldberg,
  author    = {Goldberg, A. A. and Ostrovskki, I. V.},
  title     = {Value Distribution of Meromorphic functions,Translation of Mathematical Monographs, Vol 236},
  publisher = {American Mathematical Society},
  address   = {Providence, RI},
  year      = {2008}
}

@book{Hayman,
  author    = {Hayman, W. K.},
  title     = {Meromorphic Functions, Oxford Mathematical Monographs},
  publisher = {Clarendon Press},
  address   = {Oxford},
  year      = {1964}
}

@article{Qiao,
  author  = {Qiao, J.},
  title   = {On limiting directions of {Julia} sets},
  journal = {Ann. Acad. Sci. Fenn. Math.},
  volume  = {26},
  year    = {2001},
  pages   = {391--399}
}

@article{Rempe2014,
  author  = {Rempe-Gillen, L.},
  title   = {Hyperbolic entire functions with full hyperbolic dimension and approximation by {Eremenko--Lyubich} functions},
  journal = {Proc. Lond. Math. Soc. (3)},
  volume  = {108},
  year    = {2014},
  pages   = {1193--1225},
  doi     = {10.1112/plms/pdt048},
  url     = {https://doi.org/10.1112/plms/pdt048}
}

@article{RipponStallard,
  author  = {Rippon, P. J. and Stallard, G. M.},
  title   = {Fast escaping points of entire functions},
  journal = {Proc. Lond. Math. Soc. (3)},
  volume  = {105},
  year    = {2012},
  pages   = {787--820},
  doi     = {10.1112/plms/pds001},
  url     = {https://doi.org/10.1112/plms/pds001}
}

@article{WangYao,
  author  = {Wang, J. and Yao, X.},
  title   = {On {Julia} limiting directions of meromorphic functions},
  journal = {Israel J. Math.},
  volume  = {238},
  year    = {2020},
  pages   = {405--430},
  doi     = {10.1007/s11856-020-2037-5},
  url     = {https://doi.org/10.1007/s11856-020-2037-5}
}

@article{WangChen,
  author  = {Wang, J. and Chen, Z.},
  title   = {Limiting directions of {Julia} sets of entire solutions to complex differential equations},
  journal = {Acta Math. Sci. Ser. B (Engl. Ed.)},
  volume  = {37},
  year    = {2017},
  pages   = {97--107}
}

@article{ZhengWangHuang,
  author  = {Zheng, J.-H. and Wang, S. and Huang, Z.-G.},
  title   = {Some properties of {Fatou} and {Julia} sets of transcendental meromorphic functions},
  journal = {Bull. Aust. Math. Soc.},
  volume  = {66},
  year    = {2002},
  pages   = {1--8}
}

@incollection{Schleicher,
  author    = {Schleicher, D.},
  title     = {Dynamics of Entire Functions},
  booktitle = {Holomorphic Dynamical Systems},
  series    = {Lecture Notes in Mathematics},
  volume    = {1998},
  publisher = {Springer},
  address   = {Berlin},
  year      = {2010},
  pages     = {295--339},
  doi       = {10.1007/978-3-642-13171-4_5}
}

@article{EremenkoLyubich1990,
  author  = {Eremenko, A. E. and Lyubich, M. Y.},
  title   = {The Dynamics of Analytic Transformations},
  journal = {Leningrad Mathematical Journal},
  volume  = {1},
  year    = {1990},
  pages   = {563--634}
}

@article{bergweiler25,
	author = {Bergweiler, W. and Rempe, L.},
	doi = {10.1365/s13291-025-00300-1},
	fjournal = {Jahresbericht der Deutschen Mathematiker-Vereinigung},
	issn = {0012-0456,1869-7135},
	journal = {Jahresber. Dtsch. Math.-Ver.},
	mrclass = {37F10 (30D15)},
	mrnumber = {5000874},
	number = {1-4},
	pages = {1--141},
	title = {The escaping set in transcendental dynamics},
	url = {https://doi.org/10.1365/s13291-025-00300-1},
	volume = {127},
	year = {2025}}
\end{document}